\documentclass[11pt]{article}
\usepackage{amssymb}
\usepackage{mathtools}		
\usepackage{mathabx}		
\usepackage{mathrsfs}		

\usepackage[normalem]{ulem} 

\usepackage{amsthm}
\usepackage{thmtools}

\usepackage{enumitem}		
\usepackage[backref=page,colorlinks=true]{hyperref}	
\usepackage[capitalize]{cleveref} 

\usepackage[font=small]{caption}

\usepackage{tikz}

\usepackage{xcolor}
\hypersetup{
 colorlinks,
 linkcolor={red!50!black},
 citecolor={blue!50!black},
 urlcolor={blue!80!black}
}

\renewcommand*{\backref}[1]{}
\renewcommand*{\backrefalt}[4]{\quad \tiny 
 \ifcase #1 (\textbf{NOT CITED.})%
 \or (Cited on page~#2.)%
 \else (Cited on pages~#2.)%
 \fi}

\makeatletter
 \def\MRbibitem{\@ifnextchar[\my@lbibitem\my@bibitem}
 
\def\mybiblabel#1#2{\@biblabel{#2}}
 
\def\myhyperanchor#1{\Hy@raisedlink{\hyper@anchorstart{cite.#1}\hyper@anchorend}}

\def\my@lbibitem[#1]#2#3#4\par{%
 \item[\mybiblabel{#2}{#1}\myhyperanchor{#3}\hfill]#4%
 \@ifundefined{ifbackrefparscan}{}{\BR@backref{#3}}%
 \if@filesw{\let\protect\noexpand\immediate
 \write\@auxout{\string\bibcite{#3}{#1}}}\fi\ignorespaces%
}
 
\def\my@bibitem#1#2#3\par{%
 \refstepcounter\@listctr 
 \item[\mybiblabel{#1}{\the\value\@listctr}\myhyperanchor{#2}\hfill]#3%
 \@ifundefined{ifbackrefparscan}{}{\BR@backref{#2}}%
 \if@filesw\immediate\write\@auxout
 {\string\bibcite{#2}{\the\value\@listctr}}\fi\ignorespaces%
}

\makeatother

 \declaretheoremstyle[
headfont=\footnotesize\bf,
bodyfont=\footnotesize\sf
]{myremark}

\declaretheorem[Refname={Theorem,Theorems}, name=Theorem, numberwithin=section]{otherthm}
\declaretheorem[Refname={Lemma,Lemmas}, sibling=otherthm]{lemma}
\declaretheorem[Refname={Corollary,Corollaries}, sibling=otherthm]{corollary}
\declaretheorem[Refname={Proposition,Propositions}, name=Proposition, sibling=otherthm]{proposition}
\declaretheorem[Refname={Conjecture,Conjectures}, name=Conjecture, sibling=otherthm]{conjecture}
\declaretheorem[Refname={Definition,Definitions}, sibling=otherthm, style=definition]{definition}

\declaretheorem[Refname={Question,Questions}, sibling=otherthm, style=plain]{question}
\declaretheorem[Refname={Problem,Problems}, sibling=otherthm, style=plain]{problem}
\declaretheorem[Refname={Remark,Remarks}, sibling=otherthm, style=remark]{remark}

\declaretheorem[
  name=Fact,
  Refname={Fact,Facts},
  sibling=otherthm
]{fact}

\hypersetup{bookmarksdepth = 3} 
  
\numberwithin{equation}{section} 
\setlist[enumerate,1]{label={\upshape(\alph*)},ref=\alph*}
\setlist[enumerate,2]{label={\upshape(\arabic*)},ref=\arabic*}

 \newcommand{\tribar}[1]{\mathopen{| {\kern -1.5pt} | {\kern -1.5pt} |} {#1}
\mathclose{| {\kern -1.5pt} | {\kern -1.5pt} |}}

\newcommand{\qand}{\quad \text{and} \quad}

\newcommand{\se}{\mathsf e}

\newcommand{\Scl}{\mathrm{Scl} }
\newcommand{\Ord}{\mathrm{Ord}\, }
\newcommand{\Dim}{\mathrm{Dim}\, }

\newcommand{\R}{\mathbb{R}}

\newcommand{\N}{\mathbb{N}}
\newcommand{\cA}{\mathcal{A}}
\newcommand{\cD}{\mathcal{D}}\newcommand{\cE}{\mathcal{E}}

\newcommand{\cM}{\mathcal{M}}\newcommand{\cN}{\mathcal{N}}

\newcommand{\cP}{\mathcal{P}}\newcommand{\cQ}{\mathcal{Q}}
\newcommand{\cU}{\mathcal{U}}

\newcommand{\arxiv}[1]{Preprint \href{http://arxiv.org/abs/#1}{arXiv:{#1}}}
\newcommand{\doi}[1]{%
  \href{https://doi.org/#1}{\nolinkurl{DOI:#1}}%
}

\renewcommand{\phi}{\varphi}
\renewcommand{\setminus}{\smallsetminus}

\newcommand{\ord}{\mathrm{Ord}}
\newcommand{\Leb}{\mathrm{Leb}}
\newcommand{\Diff}{\mathit{Diff}}
\newcommand{\Symp}{\mathit{Symp}}

\DeclareMathOperator{\supp}{supp}

\newcommand{\Eme}{\mathscr{E}} 
\newcommand{\bfemph}[1]{\textbf{\emph{#1}}}

\begin{document}
\title{Emergence in dynamical systems}
\date{\today}
\author{Pierre Berger\thanks{This work was supported by the European Research Council (ERC) under the European Union's Horizon 2020 research and innovation programme (grant agreement No.~818737, Emergence).}\\[0.3em]
{\small IMJ-PRG, CNRS, Sorbonne Universit\'e, Universit\'e Paris Cit\'e}}
\maketitle
\begin{abstract}
In \cite{Be17}, we introduced the notion of emergence to quantify the statistical complexity of a dynamical system. Roughly speaking, emergence measures the number of probability measures required to describe, up to a given precision $\epsilon$, the statistical behavior of most orbits. A system is said to have high emergence when this number
grows super-polynomially as $\epsilon\to0$.

We survey recent developments in a program aimed at understanding the prevalence of high emergence in differentiable dynamics. We review several notions of emergence and their connections with quantization, ergodic decompositions, and entropy, and discuss examples exhibiting high or maximal emergence in conservative, symplectic, analytic, and dissipative dynamics, as well as in constrained families such as unimodal, Hénon, and rational maps.

 We also present new results concerning variants of metric and topological emergence, convexity properties, local emergence, and variational principles, together with a collection of open problems on the typicality of high emergence.
\end{abstract}

\belowpdfbookmark{\contentsname}{toc} 
\tableofcontents

\section*{Introduction}\label{setting}

Emergence is a buzzword in natural and social sciences. There are many studies in mathematics which interpret the notion of emergence differently. In \cite{Be17}, I proposed a definition of emergence that aims to describe the statistical complexity of differentiable dynamical systems. Basically, the idea is to quantify the number of probability measures needed to describe the statistical behavior of most (Leb.) orbits of the systems with precision $\epsilon$. 
When this number grows super-polynomially as $\epsilon\to0$,
the dynamical system is said to have \emph{high emergence}. The idea of the program was to focus on them.
Such systems are interesting because the number of statistics needed to describe the system grows so fast that it is practically impossible to describe the statistical behavior of the systems. I also proposed a program aiming to show the ``typicality" of dynamics with high emergence, in many senses and in many categories. We shall survey some progress in this program and introduce new results. 
 
In \cite{BB21}, with J. Bochi, we developed this concept by focusing on differentiable dynamics which preserve the Lebesgue measure, and more generally, a probability measure. We noticed that the emergence is actually the quantization number of the ergodic decomposition. We also introduced the notion of topological emergence, as the covering number of the set of ergodic measures. I will recall these concepts in \cref{section 1}. We will generalize in \cref{quantization2emergence_gen} the link between the quantization number and emergence to the non-conservative setting. We will introduce a variation of the latter notions: the second metric and topological emergences. We will show in \cref{Finite emergence} that if the second metric emergence of a system is finite, then there are finitely many measures whose basins contain a.e. point. We will finish \cref{section 1} by recalling the notion of local emergence as introduced in \cite{Be20}; it is at most the metric emergence by Helfter's theorem \cite{He25}.

In \cref{section 2}, we will present several examples of dynamics with high emergence, starting with the examples obtained with J. Bochi in \cite{BB21} in the category of conservative surface maps and their local genericity, the recent work with Turaev on the local Kolmogorov typicality of smooth symplectomorphisms with high emergence \cite{BT25}, and the recent works on analytic surface dynamics with maximal order of local emergence \cite{Be22,De25}. Then we will present three kinds of examples of dynamics with high emergence in the dissipative setting. The first kind is the dissipative counterpart of the symplectic examples (where KAM tori are replaced by attracting tori), the second is given by wandering stable components by the works \cite{KNS25,BB23}, and finally the examples are given by the construction \emph{à la} Hofbauer-Keller, see \cite{HK95,Ta22,BG25}. 

In \cref{section 3}, we will recall and develop the dictionary introduced in \cite{Be20} between the notion of emergence and entropy. Also, we will present several problems, some of which are new. 

In \cref{section 4}, we will study the convexity property of the metric and local emergence. Then we will show that the local emergence is an affine function of the measure. 
Finally, we will recall the variational principle proved in \cite{BB21} and give its counterpart for the second metric and topological emergence.

\medskip
 
\thanks I am grateful for many discussions with my colleagues which enabled me to write these notes, in particular F. Béguin, S. Biebler, J. Bochi, P.-A. Guihéneuf, M. Helfter, F. Ledrappier, Y. Nakano, A. Talebi and D. Turaev. 

\section{Definitions of Emergences}\label{section 1}
Let $(X,\mu)$ be a compact metric space $X$ endowed with a probability measure~$\mu$. Let $f$ be a measurable self-map of $X$ (not necessarily $\mu$-preserving).
The \bfemph{$n^{\mathrm{th}}$ empirical measure} of $f$ at $x\in X$ is:
\[\mathsf e^f_n(x)=\frac1n \sum_{0\le k<n} \delta_{f^k(x)}\; .\] 
If the limit exists, we denote it by $\mathsf e^f(x):= \lim_{n\to \infty} \mathsf e^f_n(x)$. 
If the limit exists $\mu$-a.e., we say that $(f,X,\mu)$ is \bfemph{empirical}. Then the limit $\mathsf e^f$ is called the \bfemph{empirical function} of $f$. By Birkhoff's ergodic theorem,
 any system which leaves the measure $\mu$ invariant ($f_*\mu = \mu$) is empirical.

We endow the space $\cM(X)$ of probability measures on $X$ with the Kantorovich-Wasserstein distance $\mathsf{d}$:
\[ \mathsf{d}(\mu_1, \mu_2)= \sup_{\phi\in \text{Lip}^1(X, \R)}\int_X \phi \, d(\mu_1-\mu_2)=\inf_{\hat \mu \in \tau(\mu_1, \mu_2) } \int_{X\times X} \mathsf{d}(x,y) \, d\hat \mu(x,y) \; ,\] 
where $\tau(\mu_1, \mu_2)$ is the set of measures in $ \cM(X\times X)$ which transport $\mu_1$ to $\mu_2$: $p_{i*} \hat \mu = \mu_i$, with $p_i:(x_1,x_2)\in X^2\mapsto x_i\in X$. This distance endows $\cM(X)$ with the weak $\star$ topology.

\subsection{Emergence of general systems}
\begin{definition}[\cite{Be17}] \label{def.metric_em}
The \bfemph{metric emergence} $\Eme_{\mu}(f)$ is the function that associates to $\epsilon>0$ the minimal number $\Eme_{\mu}(f)(\epsilon)=N$ of probability measures $\mu_1$, \dots, $\mu_N$ satisfying: 
\begin{equation}\label{emergence def}
\limsup_{n\to \infty} \int \min_{1\le i\le N} \mathsf{d}(\se^f_n(x), \mu_i) \, d\mu(x) < \epsilon \; .
\end{equation}\end{definition}
Let us introduce a variation of this definition:
\begin{definition} \label{def.metric_em2}
The \bfemph{$2^{nd}$-metric emergence} $ \Eme'_{\mu}(f)$ is the function that associates to $\epsilon>0$ the minimal number $\Eme'_{\mu}(f)(\epsilon)=N$ of probability measures $\mu_1$, \dots, $\mu_N$ satisfying: 
\begin{equation}\label{emergence def2}
 \int \varlimsup_{n\to \infty} \min_{1\le j\le N} \mathsf{d}(\se^f_n(x), \mu_j) \, d\mu(x) < \epsilon \; .
\end{equation}\end{definition}
We notice that:
\begin{equation} \Eme_\mu(f)\le \Eme'_\mu(f).\end{equation}
Kiriki, Nakano and Soma introduced \emph{pointwise emergence}
in \cite{KNS25}. At a point $x\in X$, it corresponds to metric
emergence with respect to the Dirac measure $\mu=\delta_x$.
The size and fractal properties of sets of points
with high or prescribed pointwise emergence have been
studied by Hou, Lin and Tian~\cite{HLT23},
Hou, Lin, Tian and Zhao~\cite{HLTZ24},
Ji, Chen and Lin~\cite{JCL22},
Nakano and Zelerowicz~\cite{NZ21},
Wu~\cite{Wu26}, and Zelerowicz~\cite{Ze25}.

\begin{proposition} If $(f,X,\mu)$ is empirical,
 then the metric emergences are equal:
\[\Eme_\mu(f)= \Eme'_\mu(f)\]
\end{proposition}
\begin{proof} For any finite tuple of measures $(\mu_j)_j$, the sequence of functions $x\mapsto \min_{1\le j\le N} \mathsf{d}(\se^f_n(x), \mu_j)$ is bounded. 
Thus, by the dominated convergence theorem, we have:
\[ \int \min_{j} \mathsf{d}(\se^f , \mu_j) \, d\mu = \int \lim_n \min_{j} \mathsf{d}(\se^f_n , \mu_j) \, d\mu =
 \lim_n \int \min_{j} \mathsf{d}(\se^f_n , \mu_j) \, d\mu \; .\]
This implies the desired equality. 
\end{proof} 
A motivation for introducing the $2^{nd}$ metric emergence is the following:
\begin{otherthm}\label{Finite emergence} 
If the $2^{nd}$-metric emergence is finite:
\[N:=\limsup_{\epsilon\to 0} \cE'_\mu<\infty,\]
 then there are exactly $N$ measures whose basins cover $\mu$-a.e. point in $X$.
\end{otherthm}
\begin{proof} By compactness of $\cM(X)$, there exists $(\mu_j)_{1\le j\le N}$ such that for every $\epsilon>0$, it holds: 
\[ \int \varlimsup_{n\to \infty} \min_{1\le j\le N} \mathsf{d}(\se^f_n(x), \mu_j) \, d\mu(x) < \epsilon \; .\]
Hence $ \varlimsup_{n\to \infty} \mathsf{d}(\se^f_n(x),\{ \mu_j: 1\le j\le N\})= 0$ a.e. This means that the set of accumulation points of $(\se^f_n(x))_n$ is included in 
$\{ \mu_j: 1\le j\le N\}$. As the set of accumulation points of $(\se^f_n(x))_n$ is connected, we obtain that for $\mu$-a.e. $x$, the sequence $(\se^f_n(x))_n$ converges to a certain $\mu_j$, i.e. $x$ is in the basin of $\mu_j$. \end{proof}
\begin{remark}
If $f$ is continuous, the physical measures $(\mu_j)_j$
must be invariant, but need not be ergodic. A counterexample can be constructed with $X=[0,1]$, $\mu=\Leb$ and 
$f(x)= 4\lambda x(1-x)$ with $\lambda$ a parameter \emph{à la} Hofbauer-Keller \cite{HK90}. 
\end{remark}
\begin{definition}[\cite{BB21}] \label{def.te}
The \bfemph{topological emergence} $\Eme_\mathrm{top}(f) $ of $f$ is the function which associates to $\epsilon>0$ the minimal number of $\epsilon$-balls of $\cM(X)$ whose union covers the subset $\cM_f^\mathrm{erg}$ of ergodic measures.
\end{definition}
\begin{definition} \label{def.te2}
The \bfemph{$2^{nd}$ topological emergence} $ \Eme'_\mathrm{top}(f) $ of $f$ is the function which associates to $\epsilon>0$ the minimal number of $\epsilon$-balls of $\cM(X)$ whose union covers the subset $\bigcup_{x\in X} \mathrm{Acc}(\mathsf e^f_n(x))_n$ of all the accumulation points of the empirical measures.
\end{definition}
The following is immediate:
\begin{proposition} $\Eme_{\mathrm{top}}(f)\le \Eme'_{\mathrm{top}}(f)$. 
\end{proposition}
We proved using Birkhoff's ergodic theorem:
\begin{proposition}[\cite{BB21}]\label{comparision metric top}
 If $f_*\mu =\mu$ then the metric emergence is at most the topological emergence:
\[\Eme_{\mu}(f)\le \Eme_{\mathrm{top}}(f)\; .\]
\end{proposition}
Similarly, we have:
\begin{proposition} For any $\mu$ (not necessarily invariant), the $2^{nd} $-metric emergence is at most the $2^{nd}$ topological emergence:
\[ \Eme'_{\mu}(f)\le \Eme'_{\mathrm{top}}(f)\; .\]
\end{proposition}
\begin{proof} 
Let $\epsilon>0$ and $N:=\Eme'_{\mathrm{top}}(f)(\epsilon)$.
Choose $\mu_1,\ldots,\mu_N\in\cM(X)$ such that the open balls
$B(\mu_i,\epsilon)$ cover
$\bigcup_{x\in X}\mathrm{Acc}(\mathsf e_n^f(x))_n$.
 For every $x\in X$, the set
$K_x:=\mathrm{Acc}(\mathsf e_n^f(x))_n$ is compact.
Since it is covered by the open balls $B(\mu_i,\epsilon)$,
we have
\[
\limsup_{n\to\infty}
\mathsf d(\se_n^f(x),\{\mu_i:1\le i\le N\})
=
\max_{\nu\in K_x}
\mathsf d(\nu,\{\mu_i:1\le i\le N\})
<\epsilon.
\]
%
%
 Thus:
\[ \int \varlimsup_{n\to \infty} \min_{1\le i\le N} \mathsf{d}(\se^f_n(x), \mu_i) \, d\mu(x) < \epsilon \; .\]
From which we deduce that $ \Eme'_{\mu}(f)(\epsilon)\le N= \Eme'_{\mathrm{top}}(f)$.\end{proof} 

\begin{remark}In the example of Bowen eyes, with $X$ the closure of the eye and $\mu$ the Lebesgue measure, we have $ \Eme'_{\mu}(f)> \Eme_{\mathrm{top}}(f)$, see \cite{Ga92}.
 \end{remark}
 To summarise, we have:
 \[ \boxed{
 \Eme_{\mu}(f) \le \Eme'_{\mu}(f) \le \Eme'_{top}(f) \ge \Eme_{top}(f) \; .}\]
 \[ \boxed{f_*\mu = \mu\quad \Rightarrow \quad 
 \Eme_{\mu}(f) = \Eme'_{\mu}(f) \le 
 \Eme_{top}(f) \le \Eme'_{top}(f) \; .}\]

\subsection{Metric emergence, ergodic decomposition and quantization number} 
When $(f,X,\mu)$ is empirical, the pushforward
$\mathsf e^f_*\mu\in\cM(\cM(X))$ is the distribution
of its empirical limits. When $\mu$ is $f$-invariant,
this distribution is the \emph{ergodic decomposition of $\mu$}. It is a measure on the set $\cM(X) $ of the measures on $X$. 
In this case, in \cite{BB21}, we linked the notion of emergence to a geometrical property of the ergodic decomposition. 
To this end, we recall that given a compact metric space $Y$, the \bfemph{quantization number} of a probability measure $\mu\in \cM(Y)$ is the minimal number $ N$ of points $\{x_i: 1\le i\le N\}$ such that $\mu$ is at distance less than $\epsilon$ from the convex hull of $\{\delta_{x_i}: 1\le i\le N\}$, where $\cM(Y)$ is endowed with the Kantorovich-Wasserstein distance induced by the metric of $Y$. By taking $Y=\cM(X)$, the quantization number can be evaluated for measures in $\cM(\cM(X))$, such as the ergodic decomposition. Then we obtained the following characterization:
 \begin{otherthm}[\cite{BB21}]\label{quantization2emergence}
If $(f,X,\mu)$ is empirical, then the metric emergence is the quantization number of $\mathsf e^f_*\mu$.
\end{otherthm}
 In other words, this theorem states that the metric emergence $\mathcal E^f_\mu(\epsilon)$ at scale $\epsilon$ is equal to the minimum number of measures $(\mu_i)_{1\le i\le N}$ such that with $H$ the convex hull of $\{\delta_{\mu_i}: 1\le i\le N\}$, it holds:
\[ \mathsf d( {\se}_*^f\mu , H)<\epsilon\; .\]
When $f$ is non-empirical, we can define similarly:
\begin{definition} For every $n\ge 1$, let $\cQ_\epsilon (\se^f_{k*} \mu )_{k\ge n} $ be the minimum number of measures $(\mu_i)_{1\le i\le N}$ such that with $H$ the convex hull of $\{\delta_{\mu_i}: 1\le i\le N\}$, it holds:
\[\sup_{k\ge n} \mathsf d( {\se}_{k*}^f\mu , H)<\epsilon\; .\]
 \end{definition} 
 We have the following immediate generalization of \cref{quantization2emergence}:
 \begin{otherthm} \label{quantization2emergence_gen}
It holds:
\[ \cE_\mu (f)(\epsilon) = \lim_{n\to \infty } \cQ_\epsilon ( \se^f_{k*} \mu )_{k\ge n}\; .\] 
\end{otherthm}
\begin{proof} Given a finite set $F= \{ \mu_j: j\}$ of measures, we have:
\[\int \min_j \mathsf d( \se^f_k, \mu_j )d\mu<\epsilon \Leftrightarrow \int \mathsf d( \nu , F )d\se^f_{k*} \mu<\epsilon .\]
 Then by the proof of \cite[Prop 3.2]{BB21}, this is equivalent to saying that $\se^f_{k*} \mu$ is included in the $\epsilon$-neighborhood of the convex hull of $\{\delta_{\mu_j} : \mu_j\in F\}$. Hence 
$\sup_{k\ge n} \int \min_j \mathsf d( \se^f_k, \mu_j )d\mu<\epsilon$ is equivalent to saying that 
$cl\, \{ \se^f_{k*} \mu: k\ge n\}$ is included in the $\epsilon$-neighborhood of the convex hull of $\{\delta_{\mu_j} : \mu_j\in F\}$. The desired equality is obtained by taking the limit as $n\to \infty$ and then taking $F$ of minimal cardinality. \end{proof}
\subsection{Dimension and order of emergence} 
When studying the dimensions of a metric space or of a probability measure, one studies the polynomial rate of growth of functions describing its geometry.
Hence, given a function $\phi:(0, \infty) \to (0, \infty) $, we define the upper and lower dimensions of $\phi$ as:
\[\overline{\Dim} \phi:= \varlimsup_{\epsilon\to 0} \frac{ |\log \phi(\epsilon)|}{|\log \epsilon|}
\qand
\underline{\Dim} \phi:= \varliminf_{\epsilon\to 0} \frac{ |\log \phi(\epsilon)|}{|\log \epsilon|}\; .
\]
Then the upper and lower box dimensions of $X$ are:
\[\overline{\Dim}_B X :=\overline{\Dim} N^B_X \qand \underline{\Dim}_B X :=\underline{\Dim} N_X^B \; ,\] 
where $N^B_X(\epsilon)$ is the $\epsilon$-covering number of $X$. Likewise, given a probability measure $\nu$ on $X$ and $x\in X$, we can define the local dimensions at $x$ as:
\[\overline{\Dim}_{loc}  \nu (x) :=\overline{\Dim} (\epsilon \mapsto \nu(B(x,\epsilon)))\qand 
 \underline{\Dim}_{loc}  \nu (x) :=\underline{\Dim} (\epsilon \mapsto \nu(B(x,\epsilon)))\; .\]
When studying the empirical functions, we deal with subsets or measures on the space of probability measures. This space is infinite dimensional. The \emph{order} is often more suitable. Given a function $\phi:(0, \infty) \to (0, \infty) $, we define the upper and lower orders of $\phi$ as:
\[\overline{\Ord} \phi:= \varlimsup_{\epsilon\to 0} \frac{\log |\log \phi(\epsilon)|}{|\log \epsilon|}
\qand
\underline{\Ord} \phi:= \varliminf_{\epsilon\to 0} \frac{\log |\log \phi(\epsilon)|}{|\log \epsilon|}\; .
\]
By convention, both orders are zero if $\phi(\epsilon)=1$
for all sufficiently small $\epsilon$.
Likewise, given a subspace $Y\subset \cM(X)$ and a probability measure $\nu $ on $\cM(X)$, we define:
\[\overline{\Ord}_B Y :=\overline{\Ord} N^B_Y \qand \underline{\Ord}_B Y :=\underline{\Ord} N_Y^B\; .\] 
\[\overline{\Ord}_{loc}  \nu :=\overline{\Ord} (\epsilon \mapsto \nu(B(\cdot ,\epsilon)))\qand 
 \underline{\Ord}_{loc}  \nu :=\underline{\Ord} (\epsilon \mapsto \nu(B(\cdot ,\epsilon)))\; .\]
 The following confirmed that the order is well suited to studying the space of measures when $X$ is finite dimensional. 
\begin{otherthm}[\cite{BB21} with R. Peyre]\label{Bound topo} It holds:
\[\underline \Dim_B X\le \underline \ord_B \cM(X) 
\le \overline \ord_B \cM(X) \le \overline \Dim_B X\; .\]
\end{otherthm}
When $\underline \Dim_B X=\overline \Dim_B X$, this means that the covering number of $ \cM(X) $ is ``roughly speaking'' $\exp( \epsilon^{-\Dim_B X})$. An immediate consequence of the latter theorem is:
\begin{corollary} For any system, it holds:
\[ \overline{\Ord} \Eme_\mathrm{top}(f) \le \overline{\Ord} \Eme'_\mathrm{top}(f) \le \overline \Dim_B X\; .\]
\end{corollary}

Interestingly, the box order of the set of invariant measures is maximal in many cases:
\begin{otherthm}[\cite{BB21}]\label{Bound topo2} 
Let $f$ be a $C^{1+}$-mapping of a manifold which admits a basic hyperbolic set $K$ of box dimension $d$. 
Assume that $f$ is conformal expanding or that $f$ is a conservative surface diffeomorphism. Then the order of topological emergence of $f|K$ is $d$:
\[ 
 \underline \ord\, \cE_{top} 
= \overline \ord\, \cE_{top} = d\; .\]
\end{otherthm}

\subsection{Local emergence} 
Given an empirical system $(f,X,\mu)$ we are now ready to define:

\begin{definition} \label{def.le}
The \bfemph{local emergence dimensions} of $f$ are:
\[ \overline \Dim\cE^{loc}_\mu(f) := \int \overline \Dim_{loc}\, \se_*\mu\, d\se_*\mu \qand 
 \underline \Dim\cE ^{loc}_\mu (f) := \int \underline \Dim_{loc} \, \se_*\mu \, d\se_*\mu \; .\] 
 
 The \bfemph{local emergence orders} of $f$ are:
\[ \overline \Ord\cE^{loc}_\mu(f) := \int \overline \Ord_{loc}\, \se_*\mu\, d\se_*\mu \qand 
 \underline \Ord\cE ^{loc}_\mu (f) := \int \underline \Ord_{loc} \, \se_*\mu \, d\se_*\mu \; .\] 
\end{definition}
 Helfter solved \cite[Pbm 4.22]{Be20} by showing:
\begin{otherthm}[{\cite[Thm C and F]{He25}}]
The local emergences are at most the metric emergences:
\[ \overline \Dim\cE^{loc}_\mu(f) \le \overline \Dim\cE_\mu(f) \qand \underline \Dim\cE^{loc}_\mu(f) \le \underline \Dim\cE_\mu(f) \; .\]
\[ \overline \Ord\cE^{loc}_\mu(f) \le \overline \Ord\cE_\mu(f) \qand \underline \Ord\cE^{loc}_\mu(f) \le \underline \Ord\cE_\mu(f) \; .\]
\end{otherthm}

\section{Examples of dynamics with high emergence}\label{section 2}
In this section we will focus on the case where $(X,\mu)= (M,\Leb)$ is a compact manifold endowed with a volume form $\Leb$, and $f$ is a smooth map of $X$, as it was originally studied in \cite{Be17}. 

 \subsection{A program on typicality of high emergence} 
In \cite{Be17}, we introduced the notion of emergence to quantify how hard it is to describe the statistical behavior of a differentiable dynamical system. One way to describe it is to look at the number of probability measures needed to describe this behavior with precision $\epsilon$. If this number grows super-polynomially as $\epsilon\to0$,
we say that the emergence is high:
 \begin{definition} 
The emergence of $f$ is \emph{high} if $\overline \Dim\, \cE_\Leb(f) =\infty$. 
\end{definition} 
Equivalently, the emergence of $f$ is high if  $\cE_\Leb(f)(\epsilon)$ admits no polynomial upper bound
in $\epsilon^{-1}$ as $\epsilon\to0$. We notice that the emergence is a lower bound on the complexity (both in space and in time) to approximate numerically a dynamical system by statistics. Following the
celebrated Cobham’s thesis, an algorithm which is super-polynomial is – in practice – not feasible \cite{Co65}.

\begin{conjecture}[\cite{Be17}]\label{mainconj} High emergence is typical among many classes of differentiable dynamical systems.
\end{conjecture} 
 There are many ways to work on this conjecture, as there are both many classes of dynamical systems and ways to define ``typicality" among differentiable dynamical systems. 
 
 When the class $\cU$ is ``flexible", for instance when $\cU$ is an open subset of a Fréchet space of smooth dynamics, we studied the following notions of typicality:
 \begin{itemize}
 \item \emph{Topological genericity}: the property holds true on a countable intersection of open-dense subsets of $\cU$. 
 \item \emph{Kolmogorov typicality}: for every compact manifold $\cP$, with $\cU_\cP$ the space of $C^\infty$-families of $(f_p)_{p\in \cP}$ of maps $f_p\in \cU$, a topologically generic family $(f_p)_{p\in \cP}\in \cU_\cP$ satisfies the property for Lebesgue a.e. $p\in \cP$. 
 \end{itemize} 
 When the class is more rigid (e.g. polynomials of some degree, Euler flow), the mere existence of a system realizing the property is evidence of the typicality of this property.

 So far the following classes have been studied: smooth dynamics, symplectic dynamics, analytic pseudo-rotations, real unimodal maps, real Hénon maps, complex rational functions and complex Hénon maps. We recall the statement of these results. 
 \subsection{Examples of dynamics with polynomial emergence} 
 Let us recall examples from \cite{Be17,BB21} where $(X,\mu)=(M,\Leb)$. 
 
By definition, if the system $(f,M,\Leb) $ is ergodic, then the emergence is minimal: $\cE^f_\Leb (\epsilon)=1$ for every $\epsilon$. Examples of ergodic systems encompass irrational rotation or the doubling angle map on the circle. When a system is moreover uniquely ergodic, then the topological emergence is minimal: $\cE^f_{top} (\epsilon)=1$. 

 Among non-conservative systems, let us mention that any Axiom A diffeomorphism $f$ displays finitely many measures $(\mu_j)_{j\le N}$ whose basins cover a.e. point of $X$. Then the emergence is finite: $\lim_{\epsilon\to 0} \cE^f_\Leb (\epsilon)\le N$. 
 
 Let us mention also that the identity of $M$ satisfies that $\se^{id}_x=\delta_x$, and since $x\mapsto \delta_x$ is an isometric inclusion, we obtain that $\Dim \cE^{id} _\Leb =\Dim M$. 
 
 When $f$ is a symplectomorphism of $M^{2n}$ displaying KAM tori, then it displays an invariant  subset which is symplectomorphic to $K\times \mathbb T^n$ with $K\subset \R^n$ of positive Lebesgue measure, and the dynamics leaves invariant each $\{k\}\times \mathbb T^n $, $k \in K$.  Then the emergence dimension is at least $\Dim K=n$.
 
 \subsection{Typicality of symplectomorphisms with high emergence} 
 The first example of dynamics with high emergence was given in \cite{BB21}:
 \begin{otherthm}[\cite{BB21}]\label{thmBB1} There exists a  $C^\infty$-symplectic flow $(\phi^t)_t$ on the cylinder $\mathbb A= \mathbb T^1\times [0,1]$ whose emergence has maximal order:
 \[ \underline \Ord\, \cE_\Leb(\phi^1) =2\, .\] 
 \end{otherthm} 
 The flow $\phi^t$ is actually conjugated to a map of the form $h^{-1} \circ R_{\theta(y)} \circ h$, where $\theta:[0,1]\to \R$ vanishes at $0$ and is ``extremely" flat at $0$, while $h$ sends the Lebesgue measure supported by circles $\mathbb T\times \{y\}$ to measures which are far apart, so that the emergence is indeed of order 2.  
 
 Using the density of dynamics displaying a periodic spot\footnote{an open subset formed by periodic points.} \cite{GST} and KAM's theorem, we proved that a set of positive measure of the latter construction appears for generic conservative maps displaying elliptic periodic points:
 \begin{otherthm}[\cite{BB21}]\label{thmBB2} 
 Let $(S, \Leb)$ be a compact surface endowed with a volume form, and let $\cU$ be the open subset of $\Diff^\infty_{\Leb}(S)$ formed by dynamics displaying an elliptic periodic point. Then a generic map $f\in \cU$ has maximal emergence order:
 \[ \overline \Ord\, \cE_\Leb(f) =2\, .\] 
 \end{otherthm} 
 This theorem is stronger than the previous theorem because we have a generic set of maps with high emergence, but it is weaker because it deals with discrete time and with $\overline \Ord$ instead of $\underline \Ord$. 
 
 Recently we proved:
 \begin{otherthm}[\cite{BT25}]\label{thmBT} 
 Let $(M, \omega)$ be a compact symplectic manifold, and let $\cU$ be the open subset of $\Symp^\infty (M)$ formed by symplectomorphisms displaying a non-degenerate totally elliptic periodic point\footnote{A periodic point whose eigenvalues are of modulus 1, non-resonant up to the order 4, and such that the twist condition is satisfied   by the third derivative.}. Then Kolmogorov typically, a symplectomorphism in $\cU$ has maximal emergence order:
 \[ \overline \Ord\, \cE_\Leb(f) =\Dim M\, .\] 
 \end{otherthm} 

\subsection{Analytic pseudo-rotations} 
All the above examples are smooth and regard metric emergence.  The following is the first example of an \emph{analytic} symplectomorphism with high emergence. Moreover, it is the first example of dynamics with high local emergence.
It was found recently by disproving a conjecture of Birkhoff~\cite{Bi41}. It shows that even in the very constrained class of analytic pseudo-rotations of the cylinder (symplectomorphism without periodic point),  maximal local emergence might occur. 
\begin{otherthm}[\cite{Be22}]\label{thmBe22}
There exists an analytic symplectomorphism $f$ of the cylinder $\mathbb A= \mathbb T^1\times [0,1]$, without periodic point, and  whose order of local emergence is maximal:
 \[ \overline \Ord\, \cE_\Leb^{loc} (f) =2\, .\] 
\end{otherthm} 
This theorem was proved by overlapping the patterns  of \cref{thmBB1} at multiple scales, using the approximation by conjugacy method of Anosov-Katok. A main issue was to implement this method in the analytic setting. I proposed an AbC principle to apply this method on the sphere and the disk in \cite{Be24}. This principle made it possible to solve a conjecture of Birkhoff on the instability of elliptic points \cite{Bi27} and has been applied by Delaporte to obtain:
 \begin{otherthm}[\cite{De25}] \label{thmDe22}
There are analytic symplectomorphisms $f$ of the sphere and of the disk, with two and one periodic points, respectively, whose order of local emergence is maximal:
 \[ \overline \Ord\, \cE_\Leb^{loc} (f) =2\, .\] 
\end{otherthm} 

Finally, let us mention the work of Carvalho-Rodrigues-Varandas \cite{CRV24} which showed that a topologically generic homeomorphism of a manifold $M$ of dimension $\ge 2$ has topological emergence of maximal order $\Dim M$.

 \subsection{Periodic flows}
A flow is \emph{periodic} if all its orbits are circles.
Examples with unbounded periods were constructed on compact
manifolds of dimensions $5$ and $4$ in \cite{Su76,EV78},
respectively. Such dynamics are far from being ergodic.

In \cite[Chap.~2]{Co25}, we showed that Thurston's and
Epstein--Vogt's examples
have low (polynomial) emergence, using a general criterion.
On the other hand, Costa~\cite{Co26} constructed a variation
of the Epstein--Vogt example~\cite{EV78}, whose nonzero time
maps have upper emergence order at least $2$, using
the high-emergence mechanism of \cite{BB21}.

 \subsection{Examples of differentiable dynamics with high emergence} 
Among dissipative systems, for now, there are three kinds of examples. 

 \subsubsection{Persistent attracting circles} The first kind is the dissipative counterpart of \cref{thmBB2} where the persistence of KAM tori is replaced by persistent attracting circles, and the existence of periodic spots is provided by Turaev's Theorem \cite{Tu15} on the existence of periodic spots within the total Newhouse domain\footnote{Dynamics displaying a horseshoe with a persistent homoclinic tangency, and containing two periodic points which are resp. area contracting and expanding. } $\cN$ :
 
 \begin{otherthm}[\cite{BB21}]\label{thmBB3} 
 Let $(S, \Leb)$ be a compact surface endowed with a volume form. Then there is a nonempty open subset $\cN\subset \Diff^\infty (S)$ such that a generic map $f\in \cN$ has maximal emergence order:
 \[ \overline {\mathrm{Ord}}\, \cE_\Leb(f) =2\, .\] 
 \end{otherthm}
 \subsubsection{Wandering stable components}
A second kind of examples goes back to a phenomenon discovered by Colli-Vargas \cite{CV01} where they found a $C^r$-diffeomorphism which displays a \emph{stable} domain $U$:
\[ 
d(f^n(x), f^n(y))\to 0 \quad \forall x,y\in U
\]
 which moreover accumulates onto an affine horseshoe. In \cite{KS17}, Kiriki and Soma showed that this phenomenon appears $C^r$-densely in the dissipative Newhouse domain $\cN^r$ of $C^r$-surface diffeomorphisms, for every $2\le r<\infty$. In \cite{KNS25}, Kiriki-Nakano-Soma showed that the empirical measures of points in a domain accumulate on an infinite-dimensional set of measures. Hence, the emergence is high for a dense subset of the dissipative Newhouse domain, for $2\le r<\infty$. With Biebler we showed that this result is also valid for the infinitely smooth case $r=\infty$ and analytic $r=\omega$. Moreover, we showed that the emergence order is positive:
\begin{otherthm}[\cite{BB23}] For every $r\in [ 2,\infty] \cup\{\omega\}$, for every surface $M$ endowed with a volume form $\Leb$, there exists a dense subset $\cD$ of the dissipative Newhouse domain $\cN^r\subset \Diff^r(M)$ such that for every $f\in \cD$ there exists a domain $U\subset M$ satisfying:
\[ d(f^n(x), f^n(y))\to 0 \quad \forall x,y\in U\qand \mathrm{Ord}_B Acc (\se^f_n(x))_{n>0} >0\quad \forall x\in U\; .\] 
In particular, the emergence order of $f$ is positive.
\end{otherthm}
Actually, this phenomenon holds in   any ``nondegenerate'' unfolding in the dissipative Newhouse domain (with 5 parameters). In particular, it appears with the family of real Hénon maps of degree 6:
\[(z,w)\mapsto (z^6 +\sum_{0\le i\le 5} a_i \cdot z^i-b \cdot w, z)\]
This shows that high emergence occurs in the very constrained family of Hénon maps of degree 6. 
 
Interestingly, this stable domain has a complex extension which is also stable and so is included in the Fatou set. Its Fatou component turns out to be wandering, that is non-periodic. This gives:
\begin{otherthm}[\cite{BB23}]\label{BB23}There exists a locally dense subset of real Hénon maps of degree 6 formed by maps displaying a Fatou component $\tilde U\subset \mathbb C^2$, which intersects $\R^2$ and such that:
\[ d(f^n(x), f^n(y))\to 0 \quad \forall x,y\in \tilde U\qand \mathrm{Ord}_B Acc (\se^f_n(x))_{n>0} >0\; .\] 
\end{otherthm}
This theorem actually provided the first example of a polynomial automorphism of $\mathbb C^2$ with a wandering Fatou component, a problem which goes back to \cite{Mi90,BS91}.

Let us mention that since then, other examples of wandering stable components have been found in \cite{Ba21, KNS23}, where the horseshoe with a persistent homoclinic tangency {\it à la} Newhouse is substituted by other mechanisms such as   horseshoes with a persistent homoclinic tangency {\it à la} Bonatti-Diaz.

 \subsubsection{Maximal oscillation}
 A third kind of examples goes back to the examples of Hofbauer and Keller \cite{HK90,HK95}. They proved the existence of an uncountable set of parameters $\lambda \in (0,1)$ such that the map $f: x\in (0,1) \mapsto 4\lambda x(1-x) $ has \emph{maximal oscillation}. This means that $\Leb$-a.e. point $x\in (0,1)$ has its sequence of empirical measures $(\se_n^f(x))_n$ which accumulates simultaneously \emph{every} invariant measure of $f$. We showed that such parameters are topologically generic in the real bifurcation locus of non-renormalizable maps, and that  such maps leave invariant hyperbolic compact subsets of Hausdorff dimension arbitrarily close to $1$. Then from \cref{Bound topo2}, we deduced:
\begin{otherthm}[\cite{BG25}]
There exists a subset $\mathrm{Bif}_{NR} \subset (0,1)$ of positive Lebesgue measure such that for a topologically generic $\lambda\in \mathrm{Bif}_{NR}$, the map $f: x\in (0,1) \mapsto 4\lambda x(1-x) $ has maximal order of emergence and maximal oscillation.
\end{otherthm}
This result was inspired by Talebi's complex counterpart to Hofbauer-Keller's theorem. It deals with the bifurcation locus $\mathrm{Bif}_{d}$ of rational functions of any degree $d\ge 2$. This set can be defined as the closure of strictly post-critically finite rational functions of degree $d$. 
\begin{otherthm}[\cite{Ta22}]
For every $d\ge 2$, a topologically generic $f\in \mathrm{Bif}_{d}$ has maximal oscillation.
\end{otherthm}
One can show  that the latter maps leave  (conformal) invariant hyperbolic compact subsets of Hausdorff dimension arbitrarily close to $2$, and so by \cref{Bound topo2}, their order of emergence is maximal (=2).

The role of critical points is highlighted by a recent result of Costa and Santiago \cite{CS26}: every $C^2$ immersion of the circle has zero upper order of metric emergence with respect to Lebesgue measure. Thus, within the class of $C^2$ circle maps, positive emergence order requires the presence of critical points.

\section{A dictionary and some open problems}\label{section 3}
Let us notice that Entropy and Emergence are orthogonal notions. For instance, a doubling angle map of the circle has metric entropy $\log 2$ while its metric emergence is minimal ($\equiv 1$). Also, the map in \cref{thmBe22} has local emergence of maximal order 2, but its topological entropy is 0 (it is actually topologically conjugate to a rotation). 

Yet the well-established theory on the entropy for differentiable dynamical systems can be used as a dictionary to develop the theory of emergence. The previous results can be linked as in \cref{tab:entropy-emergence} \cpageref{tab:entropy-emergence}.
\begin{table}[htbp]
\centering

\makebox[\textwidth][c]{%
\begin{tabular}{|c|c|}
 \hline
{Topological Entropy}: & Topological Emergence, Def. \ref{def.te} and \ref{def.te2}:\\
$\frac 1n\log$(covering $\#$ of $M$ for $d_n$)&Covering $\#$ of space of erg. measures\\
				$d_n(x,y):= \max_{k\le n} d(f^k(x), f^k(y))$					&Covering $\#$ of accumulation values of $(\se_n)_n$\\
 &\\
 \hline
Metric Entropy:& Metric Emergence, Def. \ref{def.metric_em} and \ref{def.metric_em2}:\\
$h_\mu= \sup_{\cA} \lim_n \frac1n \sum_{\cA_n }\mu(A) |\log \mu(A)|$ &$\cE_\mu(\epsilon)=\min\{N: \varlimsup \int \mathsf{d}(\se^f_n , \{\mu_i\}_{ i\le N}) \, d\mu < \epsilon\}  $\\
with $\cA_n:= \bigvee_{k< n} f^{-k}\cA$&$ \cE'_\mu(\epsilon)=\min\{N: \int \varlimsup \mathsf{d}(\se^f_n , \{\mu_i\}_{ i\le N}) \, d\mu < \epsilon\}  $\\
&\\
 \hline
Katok { Metric Entropy} of $\nu$ ergodic:& { Metric Emergence }of $\mu$ invariant Thm \ref{quantization2emergence}:
\\
$h_\nu=\lim_{n\to \infty} \frac1n \log \mathcal Q^{\mathsf d_n}_\nu$
&$\cE_\mu= \mathcal Q_{\mathsf e_* \mu}$\\
 &\\
 \hline
Margulis-Ruelle inequality : & {Theorem \ref{Bound topo}:} \\
 $h_\nu \le \Dim M\cdot \int \log^+\| Df \| d\nu$ & $\mathcal {OE}_\mu, \mathcal {OE}'_\mu \le \Dim_B X$ \\
 &\\
 \hline
Positive Entropy Conjecture :& Conjecture \ref{mainconj}:\\
Positive metric entropy is typical & High emergence is typical\\
&\\
 \hline
Ledrappier-Young Theorem 1:& Finiteness of physical measures Thm \ref{Finite emergence}:\\
$d_u =\Dim E^u \Leftrightarrow \mu \text{ SRB}$
&$\cE_\mu' \equiv N\Leftrightarrow X\stackrel{o}= \bigcup_{i\le N} B_{\mu_i}$ \\
&\\
 \hline
Ledrappier-Young Theorem 2:& Local Emergence, Def. \ref{def.le}:\\
$h_\mu = \int \sum_i \lambda_i ^+ (d_i-d_{i-1})d\mu  $
 & $\mathrm{Scl}\, \cE^{loc}_\mu  = \int \mathrm{Scl}_{loc}  \, \se_*\mu\,d(\se_*\mu)$,\\
 & for $\mathrm{Scl}\in \{\Dim, \Ord\}$ \\
 &\\
 \hline
\end{tabular}%
}
\caption{Comparison between entropy and emergence.}
\label{tab:entropy-emergence}
\end{table}
\medskip

Note that  the Kolmogorov typicality of high emergence is known for symplectomorphisms displaying elliptic points (see \cref{thmBB2,thmBT}), while in this context, the typicality of their positive metric entropy is completely open. Still,  Anosov systems provide examples of dynamics with robustly positive metric entropy. Hence it is natural to ask:
\begin{question}[Robustness of high emergence]\label{Rob_hE} Does there exist an open set of differentiable dynamics with high metric emergence?
\end{question} 
In the same spirit as Jakobson's or Benedick-Carleson's Theorems, it is also natural to ask:
\begin{question}[Abundance of high emergence]\label{Ab_hE} Does there exist an open set of families $(f_p)_{p\in \cP} $ of differentiable dynamics such that for a set of parameters $p$ of positive Lebesgue measure, the map $f_p$ has high emergence?
\end{question} 
These questions can be asked in many classes of dynamics (symplectic, dissipative, polynomial, etc). A positive answer to \cref{Rob_hE} implies a positive answer to \cref{Ab_hE}. A positive answer to these questions would have revolutionary consequences. This would mean that for some systems governed by mathematical-physical laws, the same experiment may produce results whose averages are not the same.

In the same spirit as \cref{mainconj,thmBT}, we can specify:
\begin{problem} Show the existence of an open subset $U\subset \Diff^\infty(M)$ in which high emergence is Kolmogorov typical.
\end{problem} 
\begin{problem} Show the existence of an open subset $U\subset \Diff^\infty_{\Leb} (M)$ in which high emergence is Kolmogorov typical.
\end{problem} 

We can also try to push forward the direction of \cref{thmBe22,thmDe22,BB23} on the existence of high emergence within constrained classes of dynamical systems  by asking the following questions:
\begin{problem}
$\bullet$ Does there exist a steady Euler flow with high emergence? See \cite[Pbm 2.2]{BFPS23}

$\bullet$ Can the Chirikov standard map have high emergence?

$\bullet$ Does there exist a conservative Hénon map with high emergence?

$\bullet$ Is there a planar $C^1$-horseshoe with positive Lebesgue measure and high emergence? See \cite{Bo75}.

$\bullet$ Do the dynamics of some cosmological models (Einstein field equations) in the vicinity of their initial singularity display high emergence? In particular, is this the case for the Wainwright-Hsu vector field? See \cite[Ques. 3]{BD23}

$\bullet$ Under which conditions do dynamical models associated with multiplicative
cascades~\cite{Bertoin08,BJM10} have positive order of local emergence?
\end{problem} 
 
Another problem is about the typicality on which the conclusion of \cref{thmBe22} happens:
\begin{question}How typical is high local emergence?
\end{question}

In the next section we will bring other analogies to the above dictionary:
\begin{center} 
\begin{tabular}{|c|c|}
 \hline 
Convexity:& \cref{cvxem1}, for $\mathrm{Scl}\in \{\Dim, \Ord\}$:\\
 $\mu= t_o\cdot \mu_o+t_1\cdot \mu_1 $& $ \overline{\mathrm{Scl}}\, \cE_\mu = \max_i \overline{\mathrm{Scl}} \, \cE_{\mu_i}$ and $ \overline{\mathrm{Scl}}\, \cE'_\mu = \max_i \overline{\mathrm{Scl}} \, \cE'_{\mu_i}$\\
$ h_\mu = t_o\cdot h_{\mu_o}+t_1\cdot h_{\mu_1}$ &\\
& \cref{cvxem2}:\\
 & $\Scl\, \cE_{\mu}^{loc}  = t_o\cdot \Scl \, \cE^{loc}_{\mu_o} +t_1 \cdot \Scl \, \cE^{loc}_{\mu_1} $\\
&\\
 \hline 
{Variational Principle}:& {Variational Principles} Thm. \ref{varia emergence} and \ref{varia emergence2}:\\
$h_{top}=\sup_\nu h_\nu$ & $\mathcal O\cE_{top}= \max_{\mu} \mathcal O\cE_{\mu}$ \\
 & $\mathcal O\cE'_{top}= \max_{\mu} \mathcal O\cE'_\mu$\\
 &\\
 \hline
 \end{tabular}
 \end{center} 
 In the variational principle, the measures maximizing the order or dimension of emergence are not unique. Yet 
 we can wonder if there is a canonical measure of maximal local emergence, at least for some particular cases:
 \begin{problem} Is there a canonical measure of maximal local emergence for the doubling angle map of the circle?
\end{problem}
\section{Properties of emergences}\label{section 4}
In this section we denote by $(X,\mu)$  a compact metric space endowed with a probability measure $\mu$, and by $f$ a \emph{continuous} map of $X$ which does \emph{not} need to leave $\mu$ invariant.

 \subsection{Disjoint union}
 \subsubsection{Metric emergence of disjoint unions}
 We have the following, see also \cite[Thm 3 (b)]{CRV24}:
 \begin{otherthm} \label{cvxem1} 
If $\mu= t\cdot  \mu_o+(1-t)\cdot \mu_1$ with $t\in (0,1)$, then for any $\Scl\in \{\Dim, \Ord\}$, it holds: 
 \[ \overline \Scl\, \cE_{\mu }(f) = \max \{\overline \Scl\, \cE_{\mu_o}(f),\overline\Scl \, \cE_{\mu_1}(f) \} ,\]
 \[ \overline \Scl\, \cE'_{\mu }(f) = \max \{ \overline \Scl\, \cE'_{\mu_o}(f),\overline \Scl \, \cE'_{\mu_1}(f) \} \; . \] 
 \end{otherthm} 
\begin{proof} 
To show that $\overline \Scl \, \cE_{\mu}(f)$ and $\overline \Scl \, \cE'_{\mu}(f)$ are at least the above maxima, it suffices to use the following generalization of \cite[Lem. 3.18]{BB21} with $\mu'= \mu_o$ or $\mu_1$ and $s= t$ or $1-t$:
 \begin{lemma} Let $\mu'$ be a probability measure on $X$ such that $s\cdot \mu' \le \mu$. Then 
 \[\cE_{\mu'}(f)( \tfrac \epsilon s) \le \cE_\mu(f)( \epsilon) \qand \cE'_{\mu'}(f)( \tfrac \epsilon s) \le \cE'_\mu(f)( \epsilon) \quad \forall \epsilon>0\; .\] 
 \end{lemma} 
 \begin{proof} 
 Let $(\mu_i)_{1\le i\le N}$ be a minimal family such that:
 \[  \varlimsup_{n} \int \min_i d( \se_n, \mu_i)d\mu < \epsilon \quad 
\text{ or resp. } 
 \int \varlimsup_{n} \min_i d( \se_n, \mu_i)d\mu < \epsilon\; . \]
 Since $s\,d\mu'\le d\mu$, dividing by $s$ yields:
 \[  \varlimsup_{n} \int \min_i d( \se_n, \mu_i)d\mu' < \epsilon/s \quad 
\text{ or resp. } 
 \int \varlimsup_{n} \min_i d( \se_n, \mu_i)d\mu' < \epsilon/s\; . \]
 Hence $\cE_{\mu'}(f)(\epsilon/s)\le N$ or resp. $\cE'_{\mu'}(f)(\epsilon/s)\le N$. This proves the lemma as $N$ is minimal. 
 \end{proof}
 To show that $\overline \Scl \, \cE_{\mu}(f)$ and $\overline \Scl \, \cE'_{\mu}(f)$ are at most the above maxima, it suffices to apply:
 \begin{lemma} 
 \[\cE_{\mu}(f)(\epsilon) \le \cE_{\mu_o}(f)(\epsilon)+ \cE_{\mu_1}(f)(\epsilon)\qand \cE'_{\mu}(f)(\epsilon) \le \cE'_{\mu_o}(f)(\epsilon)+\cE'_{\mu_1}(f)(\epsilon) \]
 \end{lemma} 
Indeed, given minimal families $(\mu_i)_{i\le N_o}$ and $(\mu_i)_{N_o< i\le N_o+N_1}$ such that:
 \[  \varlimsup_{n} \int \min_{i\le N_o} d( \se_n, \mu_i)d\mu_o <\epsilon \qand \varlimsup_{n} \int \min_{N_o< i\le N_o+N_1} d( \se_n, \mu_i)d\mu_1 <\epsilon ,\]
by taking the convex combination of these two inequalities, we obtain: 
 \[  \varlimsup_{n} \int \min_{i\le N_o+N_1} d( \se_n, \mu_i) d\mu <\epsilon\; . \]
We proceed similarly when the $\limsup$ is after the integral. 
 \end{proof} 
\subsubsection{Local emergence of disjoint unions} 
\begin{otherthm}\label{cvxlocalEme}\label {cvxem2} Let $(f,X,\mu)$ be empirical and let $\mu_o$ and $\mu_1$ be 
such that:
\[\mu= t \cdot \mu_o+(1-t) \cdot  \mu_1\text{ with }t\in (0,1).\] 
Then for any $\Scl\in \{\Dim, \Ord\}$, it holds: 
 \[ \overline \Scl\, \cE_{\mu}^{loc} (f) = t\cdot \overline \Scl \, \cE^{loc}_{\mu_o}(f)+(1-t) \cdot \overline\Scl \, \cE^{loc}_{\mu_1}(f)\; ,\]
 \[ 
 \underline \Scl\, \cE_{\mu}^{loc} (f) = t \cdot \underline \Scl\, \cE^{loc}_{\mu_o}(f)+(1-t) \cdot \underline \Scl \, \cE^{loc}_{\mu_1}(f) 
 \; . \] 
\end{otherthm}
\begin{proof} Let $\nu := \se_* \mu$; recall that it is a probability measure on $Y:= \cM(X)$. Then the theorem is a consequence of the following:
\begin{proposition}\label{cvxlocalEme2} Let $\nu_o$ and $\nu_1$ be probability measures on a metric space $Y$, and let $\nu=t\nu_o+(1-t)\nu_1$ with $t\in(0,1)$. Then, for any $\Scl\in \{\Dim, \Ord\}$, it holds: 
 \[ \int \overline \Scl_{loc} \, \nu d\nu = t \cdot \int \overline \Scl_{loc} \, \nu_o d\nu_o+(1-t) \cdot \int \overline \Scl_{loc} \, \nu_1 d\nu_1 \; ,\]
 \[ \int \underline \Scl_{loc}\, \nu d\nu = t \cdot \int \underline \Scl_{loc} \, \nu_o d\nu_o+(1-t) \cdot \int \underline \Scl_{loc} \, \nu_1 d\nu_1 \; .\]
\end{proposition} 
Let us show this proposition. For the sake of simplicity, we denote $ \Scl_{loc} \, \nu$ for either $\overline \Scl_{loc} \, \nu$ or $\underline \Scl_{loc} \, \nu$. 
We prove below:
\begin{fact}\label{local scale mixture}
For $\nu$-a.e. $y\in Y$, it holds:
\[
\Scl_{loc}\nu(y)
=\min\{\Scl_{loc}\nu_o(y),\Scl_{loc}\nu_1(y)\}.
\]
\end{fact}
 Let $\epsilon >0$, for $i, j\in \overline \N:= \N \cup \{\infty\}$ and $\delta \in \{0,1\}$ we denote:
\[ \left\{ \begin{array}{cl} Y^\delta_i:= \{ y\in Y: \Scl_{loc}\, \nu_\delta(y) \in [i\epsilon, (i+1)\epsilon)\} &\text{ if }i<\infty\\
&\\
 Y^\delta_i:= \{ y\in Y: \Scl_{loc}\, \nu_\delta(y) =\infty\}& \text{ otherwise.}\end{array}\right. \]
Let: 
\[Y_{i,j} :=Y^0_i\cap Y^1_j\; .\]

We remark that:
\[ Y:= \bigsqcup_{i,j\in \overline \N} Y_{i,j}\; .\] 
 Let us first establish the following equality for $i=j=\infty$, and for $|i-j|\ge 2$ when at least one of $i,j$ is finite:
\begin{equation}\label{cvxij} \int_{Y_{i,j}} \Scl_{loc}\, \nu d\nu= t \int_{Y_{i,j}} \Scl_{loc}\, \nu_o d\nu_o+(1-t) \int_{Y_{i,j}} \Scl_{loc}\, \nu_1 d\nu_1\; .\end{equation} 
If $i=j=\infty$, \cref{cvxij} follows from \cref{local scale mixture}, since all three local scales equal $\infty$ for $\nu$-a.e. point of $Y_{\infty,\infty}$.
For instance, assume now that $i<\infty$ and $i+2\le j$. If $\nu(Y_{i,j})=0$, \cref{cvxij} is obvious. Otherwise, $\nu(Y_{i,j})>0$. We will show below:
\begin{lemma}\label{nu1(Y)=0} If $i+2\le j$ and $i<\infty$, then it holds $\nu_1(Y_{i,j})=0$.
\end{lemma}
Hence only the case $\nu_o(Y_{i,j})>0$ and $\nu_1(Y_{i,j})=0$ remains. 
As $ \Scl_{loc}\, \nu_o < \Scl_{loc}\, \nu_1$, \cref{local scale mixture} gives $ \Scl_{loc}\, \nu= \Scl_{loc}\, \nu_o$ for $\nu$-a.e. $y\in Y_{i,j}$. Then:
\[ \int_{Y_{i,j}} \Scl_{loc}\, \nu d\nu= \int_{Y_{i,j}} \Scl_{loc}\, \nu_o d\nu = t \int_{Y_{i,j}} \Scl_{loc}\, \nu_o d\nu_o\]
which leads to \cref{cvxij} as $\nu_1(Y_{i,j})=0$.

Finally, if $i,j<\infty$ and $|i-j|\le 1$, \cref{local scale mixture} gives 
$ |\Scl_{loc}\, \nu- \Scl_{loc}\, \nu_\delta|\le 2\epsilon$ for every $\delta\in \{0,1\}$ and for $\nu$-a.e. $y\in Y_{i,j}$. Hence it holds:
\[ \int_{Y_{i,j}} \Scl_{loc}\, \nu d\nu= t \int_{Y_{i,j}} \Scl_{loc}\, \nu_o d\nu_o+(1-t) \int_{Y_{i,j}} \Scl_{loc}\, \nu_1 d\nu_1 
+O(\epsilon \nu(Y_{i,j}))\; .
\]

Therefore, we obtain:
\[ 
 \int_{Y } \Scl_{loc}\, \nu d\nu= \sum_{i,j} \int_{Y_{i,j}} \Scl_{loc}\, \nu d\nu 
 \]
 \[ = \sum_{i,j} t \int_{Y_{i,j}} \Scl_{loc}\, \nu_o d\nu_o +(1-t) \int_{Y_{i,j}} \Scl_{loc}\, \nu_1 d\nu_1+O(\epsilon \nu(Y_{i,j})),\]
and so:
\[ \int_{Y } \Scl_{loc}\, \nu d\nu= t \int_{Y } \Scl_{loc}\, \nu_o d\nu_o + (1-t) \int_{Y } \Scl_{loc}\, \nu_1 d\nu_1 
+O(\epsilon ) \; . \]
As $\epsilon$ is arbitrarily small, we obtain the desired equality. 
 \end{proof}
\begin{proof}[Proof of \cref{nu1(Y)=0}] Recall that $i+2\le j$ and assume for the sake of contradiction that $\nu_1(Y_{i,j})>0$. 
The proof uses the notions of Hausdorff and Packing dimensions $\Dim_P$ and $\Dim_H$ when $\Scl=\Dim$, and their counterparts the Hausdorff and Packing orders $\Ord_P$ and $\Ord_H$ when $\Scl=\Ord$. We refer to \cite{He25} for definitions. Here, let us just recall that these invariants are non-decreasing real functions from the set of subsets of $Y$. We have the following:
\begin{lemma} Let $\alpha \ge 0$ and let $Y'\subset Y$ be such that $\nu_1(Y')>0$. Then:
\begin{itemize}
\item $ \Scl_P Y'\ge \alpha$ if $ \overline \Scl_{loc} \nu_1(y) \ge \alpha$ for every $y\in Y'$. 
\item $ \Scl_H Y'\ge \alpha$ if $ \underline \Scl_{loc} \nu_1(y) \ge \alpha$ for every $y\in Y'$. 
\end{itemize} 
\end{lemma}
\begin{proof} By \cite[Lem. 3.4]{He25}, the restriction $\sigma:= \nu_1|Y'$ satisfies respectively
$ \overline \Scl_{loc} \sigma(y) \ge \alpha$ and $ \underline \Scl_{loc} \sigma(y) \ge \alpha$ for 
$\nu_1$ a.e. $y\in Y'$. Since $\sigma(Y\setminus Y')=0$, Helfter's Theorem B
\cite{He25} implies respectively that
$\Scl_P Y'\ge\alpha$ or $\Scl_H Y'\ge\alpha$. 
\end{proof} 
Hence it holds respectively that $\Scl_P Y_{i,j}\ge j\epsilon$ or $\Scl_H Y_{i,j}\ge j\epsilon$. Then we recall the following lemma:
\begin{otherthm}[Helfter] 
Let $\beta \ge 0$ and let $Y'\subset Y$. Then:
\begin{itemize} 
\item if $\overline \Scl_{loc} \nu_o(y) < \beta$ for every $y\in Y'$, then it holds $ \Scl_P Y'\le \beta$, 
\item if $\underline \Scl_{loc} \nu_o(y) < \beta$ for every $y\in Y'$, then it holds $ \Scl_H Y'\le \beta$. 
\end{itemize}
\end{otherthm}
\begin{proof} The proofs of the two items of this lemma are verbatim the ones of resp. Thm B (ii) and (iii) P30 of \cite{He25} (the proof does not use that $\nu_o( Y\setminus Y')=0$).\end{proof}  
Thus it holds respectively that $\Scl_P Y_{i,j}\le (i+1)\epsilon$ or $\Scl_H Y_{i,j}\le (i+1)\epsilon$. A contradiction with $i\le j-2$. This completes the proof of \cref{cvxij} when $|i-j|\ge 2$. \end{proof}

\begin{proof}[Proof of \cref{local scale mixture}]
The inequalities $\nu\ge t\nu_o$ and $\nu\ge(1-t)\nu_1$
give $\Scl_{loc}\nu\le\Scl_{loc}\nu_\delta$,
for $\delta\in\{0,1\}$.
We claim that equality holds $\nu_\delta$-a.e.
Otherwise, there exist rational numbers $0<a<b$ such that
\[
E:=\{y:\Scl_{loc}\nu(y)<a<b<
\Scl_{loc}\nu_\delta(y)\}
\]
has positive $\nu_\delta$-measure.
The two comparison results used in the preceding proof give
\[
b\le\Scl_P E\le a
\quad\text{or}\quad
b\le\Scl_H E\le a,
\]
for upper or lower local scales, respectively.
This is a contradiction.
Thus $\Scl_{loc}\nu=\Scl_{loc}\nu_\delta$
$\nu_\delta$-a.e. for both $\delta=0,1$,
which proves the fact.
\end{proof}

 \begin{remark} 
 The statements of \cref{cvxlocalEme} and \cref{cvxlocalEme2} are valid for any scale in the sense of \cite{He25}. 
 \end{remark}

\subsection{Partition entropies and emergences}
Partitions are understood modulo sets of $\mu$-measure zero.
Let us first recall the following:
\begin{definition}
The \emph{entropy} of a non-trivial countable partition $\cA$ of $X$ is:
\[
H(\cA):=\sum_{ A\in\cA }
\mu(A)|\log\mu(A)|\; .
\]
We define its \emph{entropy order} by:
\[
\Ord H(\cA):=\sum_{\substack{A\in\cA\\\mu(A)>0}}
\mu(A)\log|\log\mu(A)|\; .
\]
\end{definition}
Given a measurable function $\se:X\to\cM(X)$ and $\epsilon>0$, let:
\[
\Delta_\epsilon(\se)(x):=
\left|\log\mu\{x'\in X:d(\se(x),\se(x'))<\epsilon\}\right|\; .
\]
Let $f$ be a measurable self-map of $X$ which leaves $\mu$ invariant, and put
$\nu:=(\se^f)_*\mu$ and $Y:=\supp\nu\subset\cM(X)$ the smallest closed subset of full $\nu$-measure.
We assume that $Y$ contains at least two points. We will prove below: 
\begin{proposition}\label{prop small ball upper limits}
If $\overline\Ord_B Y<\infty$ then it holds:
\begin{equation}\label{ordre et Delta}
\overline\Ord\cE^{loc}_\mu(f)
=\int\varlimsup_{\epsilon\to0}
\frac{\log\Delta_\epsilon(\se^f)}{|\log\epsilon|}\,d\mu
\ge\varlimsup_{\epsilon\to0}\int
\frac{\log\Delta_\epsilon(\se^f)}{|\log\epsilon|}\,d\mu\; .
\end{equation}
If  $\overline\Dim_B Y<\infty$, then it holds:
\begin{equation}\label{dim et Delta}
\overline\Dim\cE^{loc}_\mu(f)
=\int\varlimsup_{\epsilon\to0}
\frac{\Delta_\epsilon(\se^f)}{|\log\epsilon|}\,d\mu
\ge\varlimsup_{\epsilon\to0}\int
\frac{\Delta_\epsilon(\se^f)}{|\log\epsilon|}\,d\mu\; .
\end{equation}
\end{proposition} 
Recall that by \cref{Bound topo}, $\overline\Ord_B Y<\infty$ and so \cref{ordre et Delta} hold true whenever
$\overline\Dim_B X<\infty$, in particular when $X$ is a compact
manifold.

\begin{otherthm}\label{newcreterium loc em}
Assume that $\overline\Ord_B Y<\infty$.
Let $(\cA_n)_n$ be a sequence of non-trivial finite measurable partitions of $X$ and
$\epsilon_n\to0$ such that for every $n$:
\begin{enumerate}
\item $d(\se^f(x),\se^f(x'))>\epsilon_n$ for all $x\in A$, $x'\in A'$
and distinct $A,A'\in\cA_n$;
\item $d(\se^f(x),\se^f(x'))<\epsilon_n$ for all $x,x'\in A$
and every $A\in\cA_n$.
\end{enumerate}
Then:
\[
\overline\Ord\cE^{loc}_\mu(f)
\ge\varlimsup_{n\to\infty}\frac{\Ord H(\cA_n)}{|\log\epsilon_n|}\; .
\]
If, moreover, $\overline\Dim_B Y<\infty$, then:
\[
\overline\Dim\cE^{loc}_\mu(f)
\ge\varlimsup_{n\to\infty}\frac{H(\cA_n)}{|\log\epsilon_n|}\; .
\]
\end{otherthm}
\begin{proof}
Apply \cref{ordre et Delta}, and \cref{dim et Delta} under the additional
dimensional assumption, together with the following lemma.
\end{proof}
\begin{lemma}
Let $\cA$ be a non-trivial finite measurable partition satisfying conditions (a) and (b)
above at scale $\epsilon$. Then:
\[
\int\Delta_\epsilon(\se^f)\,d\mu=H(\cA)
\qand
\int\log\Delta_\epsilon(\se^f)\,d\mu=\Ord H(\cA)\; .
\]
\end{lemma}
\begin{proof}
For $A\in \cA$ and $x\in A$, conditions (a) and (b) give:
\[
\{x'\in X:d(\se^f(x),\se^f(x'))<\epsilon\}=A\; .
\]
Thus $\Delta_\epsilon(\se^f)=|\log\mu(A)|$ on   $A$. Integrating over $X$  proves both equalities.
\end{proof}
The following implies \cref{prop small ball upper limits}:
\begin{lemma}\label{small ball upper limits}
Let $Y$ be a compact metric space and let $\nu$ be a Borel probability
measure on $Y$ which is not a Dirac measure. For $0<\epsilon<1$, set
\[
I_\epsilon(y):=-\log\nu(B(y,\epsilon))\; .
\]
Then:
\begin{enumerate}
\item If $\overline\Dim_B Y<\infty$, we have
\[
\varlimsup_{\epsilon\to0}
\frac{\int I_\epsilon\,d\nu}{|\log\epsilon|}
\le\int\overline\Dim_{loc}\nu\,d\nu\; .
\]
\item If $\overline\Ord_B Y<\infty$, we have
\[
\varlimsup_{\epsilon\to0}
\frac{\int\log I_\epsilon\,d\nu}{|\log\epsilon|}
\le\int\overline\Ord_{loc}\nu\,d\nu\; .
\]
\end{enumerate}
\end{lemma}
\begin{proof} To prove the lemma, it suffices to show that:
\[ \limsup_{\epsilon\to 0} \int F_\epsilon\, d\nu \le\int  \limsup_{\epsilon\to 0}  F_\epsilon\, d\nu,\] 
where:
\[
F_\epsilon:=\frac{I_\epsilon}{|\log\epsilon|}\quad \text{in case (a)}\quad\text{or }F_\epsilon:=\frac{\log I_\epsilon}{|\log\epsilon|}\quad \text{in case (b)}\; .
\]
The difficulty in proving the above inequality is that $F_\epsilon$ is unbounded. We are going to substitute it with $\min\{F_\epsilon,s\}$ , which will be shown to be bounded. We shall now prove that the remainder $\max\{0, F_\epsilon-s\}$ has a small integral. 

Write $N_\epsilon:=N_Y^B(\epsilon/3)$.
A cover by $N_\epsilon$ balls of radius $\epsilon/3$ yields a measurable
partition into at most $N_\epsilon$ sets of diameter less than $\epsilon$.
Each element of this partition meeting  $\{I_\epsilon>t\}$ has mass at most $e^{-t}$. Hence
\begin{equation}\label{small ball tail}
\nu\{I_\epsilon>t\}\le N_\epsilon e^{-t},\qquad t\ge0\; .
\end{equation}

For (a), choose $\overline\Dim_B Y<s'<s$. For all sufficiently small $\epsilon$, we have
$N_\epsilon\le\epsilon^{-s'}$. The preceding estimate gives
\[
\begin{aligned}
\int\max\{0,  F_\epsilon-s\}\,d\nu
&=\int_s^\infty\nu\{F_\epsilon>u\}\,du=\int_s^\infty\nu\{I_\epsilon>- u  \log \epsilon \}\,du
\\
&\le\int_s^\infty\epsilon^{u-s'}\,du
=\frac{\epsilon^{s-s'}}{|\log\epsilon|}
\longrightarrow0\; .
\end{aligned}
\]

For (b), choose $\overline\Ord_B Y<s'<s$. For all sufficiently small $\epsilon$, we have
$\log N_\epsilon\le\epsilon^{-s'}$. Therefore,
\[
\begin{aligned}
\int 
\max\{0, 
F_\epsilon-s\}\,d\nu&=\int_s^\infty \nu \{\log I_\epsilon>- u  \log \epsilon \}\,du\\
&\le\int_s^\infty\exp(\epsilon^{-s'}-\epsilon^{-u})\,du
\le\frac{\epsilon^s}{|\log\epsilon|}
\exp(\epsilon^{-s'}-\epsilon^{-s})
\longrightarrow0\; .
\end{aligned}
\]
The second inequality follows by the change of variables $v=\epsilon^{-u}$.

Consequently, when $\epsilon\to 0$:\[
\int F_\epsilon\,d\nu
=\int\min\{F_\epsilon,s\}+ \max\{0, 
F_\epsilon-s\}\,d\nu =\int\min\{F_\epsilon,s\}\,d\nu+o(1)\; .
\]
We prove below:
\begin{fact}
In both cases, the functions $\min\{F_\epsilon,s\}$ are uniformly bounded
for sufficiently small $\epsilon$.  
\end{fact}
Applying Fatou's lemma to $s-\min\{F_\epsilon,s\}$, we have:
\[
\begin{aligned} \limsup_{\epsilon\to 0} 
\int F_\epsilon\,d\nu
&= \limsup_{\epsilon\to 0}  \int\min\{F_\epsilon,s\}\,d\nu \\
&\le  \int  \limsup_{\epsilon\to 0} \min\{F_\epsilon,s\}\,d\nu   
\le  \int  \limsup_{\epsilon\to 0} F_\epsilon \,d\nu   \; .\end{aligned} \]
\end{proof}
\begin{proof}[Proof of the Fact] In case (a), we have $0\le\min\{F_\epsilon,s\}\le s$. For case (b),
since $\nu$ is not a Dirac measure, we can choose two Borel sets of positive
$\nu$-mass at positive distance from each other. Every sufficiently small
ball misses at least one of them. Consequently, there exists $q<1$ such that
\[
\nu(B(y,\epsilon))\le q
\]
for every $y\in Y$ and every sufficiently small $\epsilon$.
Thus $I_\epsilon\ge-\log q>0$, and
$\log I_\epsilon/|\log\epsilon|$ is uniformly bounded below as
$\epsilon\to0$.
\end{proof}

\subsection{Variational Principle}
With Bochi, we proved the following variational principle:
\begin{otherthm}[\cite{BB21}]\label{varia emergence} There exists an invariant probability measure $\mu$ such that:
\[\overline \Ord \Eme_{\mu}(f) = \overline \Ord\Eme_{\mathrm{top}}(f)\; .\]
\end{otherthm}

Similarly, we have:
\begin{otherthm}\label{new var principle}\label{varia emergence2} There exists a (not necessarily invariant) probability measure $\mu$ such that:
\[\overline\Ord \Eme'_{\mu}(f) = \overline\Ord \Eme'_{\mathrm{top}}(f)\; .\]
\end{otherthm}
\begin{proof} 
We apply the following theorem to a countable dense subset $Y$ of
$\bigcup_{x\in X}\mathrm{Acc}(\mathsf e_n^f(x))_n$.
The set $Y$ is Borel and has the same lower and upper box orders as this union. 
\begin{otherthm}[{\cite[thm 3.9]{BB21}}]
Let $Y$ be a Borel subset of a compact metric space $Z$. Then
there exists a probability measure $\nu \in \cM(Y)$ such that 
\[ \underline \Ord \cQ(\nu) = \underline \Ord_B Y 
\qand 
\overline \Ord \cQ(\nu) = \overline \Ord_B Y 
\; .\]
Moreover, $\nu$ is carried by a countable subset of $Y$. 
\end{otherthm}
Let $\nu:= \sum_i t_i \cdot \delta_{\nu_i}$, with $\nu_i\in Y$. 
For each $i$, let $x_i$ be such that $\nu_i\in \mathrm{Acc\, }(\mathsf e^f_n(x_i))_n$. Let $\mu:= \sum_i t_i \cdot \delta_{x_i}$.

For every $\epsilon>0$ small, let $ N(\epsilon):= \Eme'_{\mu}(f)(\epsilon)$. Hence  there are measures $(\mu_i)_{1\le i\le N}$ such that
\[
\epsilon> \int \varlimsup_{n\to \infty} \min_{1\le i\le N(\epsilon)} \mathsf{d}(\se^f_n(x), \mu_i) \, d\mu(x) 
= \sum_j t_j \cdot \varlimsup_{n\to \infty} \min_{1\le i\le N(\epsilon)} \mathsf{d}(\se^f_n(x_j), \mu_i)
 \; .
\]
Thus 
\[
\epsilon>\sum_j t_j \cdot \min_{1\le i\le N} \mathsf{d}(\nu_j , \mu_i) \; .
\]
For every $j$, let $\mu_{i_j}$ be minimizing $ \min_{1\le i\le N} \mathsf{d}(\nu_j , \mu_i)$. Hence we have:
\[
\epsilon>\sum_j t_j \cdot \mathsf{d}(\nu_j , \mu_{i_j}) \; .
\]
Recall that $ \nu:= \sum_j t_j \delta_{\nu_j}$ and put $ \hat \nu:= \sum_j t_j \delta_{\mu_{i_j}}$. Observe that the distance between $ \nu$ and $ \hat \nu$ is $<\epsilon$. Hence the quantization number of $ \nu$ satisfies $\cQ_\epsilon (\nu)\le  N(\epsilon)= \Eme'_{\mu}(f)(\epsilon)$. Thus:
\[\overline \Ord \Eme'_{\mu}(f) \ge \overline \Ord \cQ(\nu) = \overline \Ord_B Y= \overline\Ord \Eme'_{\mathrm{top}}(f) .\] 
We conclude by recalling that $ \overline\Ord \Eme'_{\mathrm{top}}(f)$ is at least $\overline \Ord \Eme'_{\mu}(f)$. 
\end{proof}

\end{document}